\documentclass[11 pt]{article}
\usepackage[top=1in, bottom=1in, left=1in, right=1in]{geometry}
\usepackage{amsmath}
\usepackage{amsthm}
\usepackage{amsfonts}
\usepackage{amssymb}
\usepackage{mathrsfs}
\usepackage{url}
\usepackage{enumerate}
\usepackage[shortlabels]{enumitem}
\usepackage{multirow}
\usepackage{float}
\usepackage{mathtools}
\usepackage{tikz-cd}

\DeclareMathOperator{\Frob}{Frob}

\DeclareMathOperator{\End}{End}
\DeclareMathOperator{\Gal}{Gal}

\DeclareMathOperator{\Ind}{Ind}
\DeclareMathOperator{\GL}{GL}
\DeclareMathOperator{\SL}{SL}

\DeclareMathOperator{\PGL}{PGL}

\DeclareMathOperator{\Tr}{Tr}
\DeclareMathOperator{\Id}{Id}

\DeclareMathOperator{\Cl}{Cl}
\DeclareMathOperator{\dimn}{dim}
\DeclareMathOperator{\Hom}{Hom}
\DeclareMathOperator{\rank}{rank}
\newcommand{\R}{\mathbb{R}}
\newcommand{\C}{\mathbb{C}}
\newcommand{\Z}{\mathbb{Z}}
\newcommand{\Q}{\mathbb{Q}}
\newcommand{\F}{\mathbb{F}}
\newcommand{\T}{\mathbb{T}}

\newcommand{\tensor}{\otimes}

\newtheorem{innercustomgeneric}{\customgenericname}
\providecommand{\customgenericname}{}
\newcommand{\newcustomtheorem}[2]{%
  \newenvironment{#1}[1]
  {%
   \renewcommand\customgenericname{#2}%
   \renewcommand\theinnercustomgeneric{##1}%
   \innercustomgeneric
  }
  {\endinnercustomgeneric}
}

\newcustomtheorem{customthm}{Theorem}
\newcustomtheorem{customlemma}{Lemma}

\newtheorem{thm}{Theorem}[section]
\newtheorem{lem}[thm]{Lemma}
\newtheorem{prop}[thm]{Proposition}

\theoremstyle{definition}

\theoremstyle{remark}

\title{On Seven Conjectures of Kedlaya and Medvedovsky}
\author{Noah Taylor\footnote{The author is supported in part by NSF Grant DMS-$1701703$.}}
\date{January 6, 2020}
\begin{document}
\maketitle
\section{Introduction}
Let $\overline{\rho}: G_{\Q}\rightarrow\GL(2, \overline{\F}_2)$ be a finite-image two-dimensional mod-2 Galois representation.  We say $\overline{\rho}$ is dihedral if the image of $\pi\circ\overline{\rho}: G_{\Q}\rightarrow\PGL(2, \overline{\F}_2)$ is isomorphic to a finite dihedral group, where $\pi: \GL(2)\rightarrow\PGL(2)$ is the usual projection.  We say $\overline{\rho}$ is modular of level $N$ if it is the reduction ${} \bmod 2$ of a representation $\rho_f$ associated to a modular eigenform $f\in S_2(\Gamma_0(N), \overline{\Z}_2)$.

We say that a characteristic $0$ representation $\rho:G_{\Q}\rightarrow\GL_2(\overline{\Q}_2)$ is ordinary at $2$ if its restriction to the inertia at $2$ is reducible.  We also say an eigenform $f$ with coefficients in $\overline{\Z}_2$ is ordinary if the coefficient $a(2)$ of $q^2$ in its $q$-expansion is a unit ${}\bmod 2$.  The terminology is consistent, because by a theorem of Deligne and Fontaine, if $\rho=\rho_f$ is modular, then $\rho_f$ is ordinary if and only if $f$ is ordinary.

In \cite{kedlaya2019}, Kedlaya and Medvedovsky prove that if a ${}\bmod 2$ representation is dihedral, modular and ordinary of prime level $N$, then it must be the induction of a character of the class group $\Cl(K)$ of a quadratic extension $K=\Q(\sqrt{\pm N})/ \Q$ to $\Q$ \cite[Section~5.2]{kedlaya2019}.  They then analyze all cases of $N\bmod 8$ to determine how many distinct ${} \bmod 2$ representations arise from this construction.  Finally, they conjecture lower bounds for the number of $\overline{\Z}_2$ eigenforms whose ${} \bmod 2$ representations $\overline{\rho}_f$ are isomorphic to each of the representations obtained above \cite[Conjecture~13]{kedlaya2019}.  The purpose of the current paper is to prove this conjecture, reproduced below.

We let $\T_2^{\text{an}}$ denote the anemic Hecke algebra inside $\End(S_2(\Gamma_0(N),\overline{\Z}_2))$ generated as a $\mathbb{Z}_2$-algebra by the Hecke operators $T_k$ for $(k, 2N)=1$, and we let $\T_2$ denote the full Hecke algebra, namely $\T_2=\T_2^{\text{an}}[T_2, U_N]$.  Maximal ideals $\mathfrak{m}$ of $\T_2^{\text{an}}$ correspond to classes of ${} \bmod 2$ eigenforms via $\T_2^{\text{an}}\rightarrow\T_2^{\text{an}}/\mathfrak{m}\xhookrightarrow{}\overline{\F}_2$, where the image of $T_k$ in the quotient is mapped to the coefficient $a(k)$ of the form.  Thus maximal ideals of $\T_2^{\text{an}}$ correspond to modular representations via the Eichler-Shimura construction.  We say that $\mathfrak{m}$ is $K$-dihedral if the representation corresponding to $\mathfrak{m}$ is dihedral in the above sense, and the quadratic extension from which it is an induction is $K$.  (Notice that given $\overline{\rho}$, $K$ is uniquely determined as the quadratic extension of $\Q$ inside the fixed field of the kernel of $\overline{\rho}$ that is ramified at all primes at which $\overline{\rho}$ is ramified.)  We write $S_2(N)_{\mathfrak{m}}$ to denote the space of all ${}\bmod 2$ modular forms on which $\mathfrak{m}$ acts nilpotently.

\begin{thm}[{\cite[Conjecture 13]{kedlaya2019}}]\label{kedmev}Let $N$ be an odd prime and $\mathfrak{m}$ a maximal ideal of $\T_2^{\textup{an}}(N)$.
\begin{enumerate}[font=\upshape]\item Suppose $N\equiv 1\bmod 8$.
\begin{enumerate}[font=\upshape]\item If $\mathfrak{m}$ is $\Q(\sqrt{N})$-dihedral, then $\dimn S_2(N)_{\mathfrak{m}}\geq4$.
\item If $\mathfrak{m}$ is $\Q(\sqrt{-N})$-dihedral, then $\dimn S_2(N)_{\mathfrak{m}}\geq h(-N)^{\textup{even}}$.
\item If $\mathfrak{m}$ is reducible, then $\dimn S_2(N)_{\mathfrak{m}}\geq\frac{h(-N)^{\textup{even}}-2}{2}$.
\end{enumerate}
\item Suppose $N\equiv 5\bmod 8$.
\begin{enumerate}[font=\upshape]\item If $\mathfrak{m}$ is ordinary $\Q(\sqrt{N})$-dihedral, then $\dimn S_2(N)_{\mathfrak{m}}\geq 4$.
\item If $\mathfrak{m}$ is $\Q(\sqrt{-N})$-dihedral, then $\dimn S_2(N)_{\mathfrak{m}}\geq 2$.
\end{enumerate}
\item Suppose $N\equiv 3\bmod 4$ and $K=\Q(\sqrt{\pm N})$.
\begin{enumerate}[font=\upshape]\item If $\mathfrak{m}$ is ordinary $K$-dihedral, then $\dimn S_2(N)_{\mathfrak{m}}\geq 2$.
\end{enumerate}
\end{enumerate}
\end{thm}

The methods we use in proving this conjecture vary somewhat among the cases listed above.  Moreover, though part $3$ is listed as a single case, we break up its proof into the cases $K=\Q(\sqrt{N})$ and $K=\Q(\sqrt{-N})$.  Thus we recognize \cite[Conjecture 13]{kedlaya2019} as $7$ separate conjectures, explaining the title of this note.

\subsection{Eigenspace dimension and modular exponent}
There is a relation between our work and the problem of understanding the parity of the modular exponent of a modular abelian variety $A=A_f$ as studied in \cite{MR2867910}.  The problems are not exactly the same, however: the dimension of $S_2(N)_{\mathfrak{m}}$ is greater than $1$ if and only if there exists two distinct eigenforms $f$ and $f'$ with $\overline{\rho}_f=\overline{\rho}_{f'}=\overline{\rho}_{\mathfrak{m}}$.  On the other hand, the modular degree is even only when there exists a congruence ${}\bmod\mathfrak{p}$ between eigenforms which are not $G_{\Q}$-conjugate, for some prime $\mathfrak{p}$ above $2$.  For example, in the case $N=29$, we know that $S_2(29)$ is $2$ dimensional, spanned by $f=q+(-1+\sqrt{2})q^2+(1-\sqrt{2})q^3+\ldots$ and $f'=q+(-1-\sqrt{2})q^2+(1+\sqrt{2})q^3+\ldots$.  These have the same ${}\bmod 2$ representation; in fact, they are even congruent ${}\bmod2$.  But the corresponding quotient of $J_0(29)$ is $J_0(29)$ itself, which is simple, so the modular exponent of these forms is $1$.

In some cases, such as when the abelian variety is an ordinary elliptic curve over $\Q$, the problems coincide, and thus this paper is related to (and generalizes) arguments from \cite{MR2460912}.  If~$A$ is a (modular) ordinary rational elliptic curve, then there is a corresponding homomorphism~$\T \rightarrow \Z$.  If~$A$ has even modular degree, then there certainly exist $2$-adic congruences between the modular eigenform~$f$ associated to~$A$ and other forms, and hence an eigenform~$f' \ne f$ with~$\overline{\rho}_f = \overline{\rho}_{f'}$.  Conversely, suppose that there exists such an~$f'$. Because~$f$ has coefficients over~$\Q$, the form~$f'$ cannot be a Galois conjugate of~$f$. Thus it suffices to show that the equality~$\overline{\rho}_{f} = \overline{\rho}_{f'}$ can be upgraded to a congruence between~$f$ and~$f'$. The only ambiguity arises from the coefficients of~$q^2$ and~$q^N$. By Theorem~\ref{deligne} below, we see that the coefficient of~$q^2$ is determined up to its inverse by the mod~$2$ representation. Yet, for~$A$, the coefficient of~$q^2$ is automatically~$1$ by ordinarity and rationality.  We also prove in Lemma~5.1 that~$U_N$ is in the Hecke algebra~$\T^{\text{an}}_2$, and thus~$f$ must be congruent to~$f'$.
\subsection{Reduction}
Given a maximal ideal $\mathfrak{m}$ of $\T_2^{\text{an}}$, we wish to count the dimension of the space $\Lambda$ of $\Z_2$-module maps\[\phi:\T_2\rightarrow \overline{\F}_2\text{ so that }\mathfrak{m}^k(\phi|_{\T_2^{\text{an}}})=0\text{ for some }k\geq0\] as an $\overline{\F}_2$-vector space, where $\T_2^{\text{an}}$ acts on $\phi$ by $x\phi(y)=\phi(xy)$.  We know that $\T_2$ and $\T_2^{\text{an}}$ are finite and flat over $\Z_2$, and thus complete semilocal rings.  It then follows that we can write \[\T_2=\bigoplus_{\mathfrak{a}\text{ maximal}}\T_{\mathfrak{a}},\]and a similar statement for $\T_2^{\text{an}}$.  We thus study $\T_{\mathfrak{m}}^{\text{an}}$ and remove the restriction that $\mathfrak{m}$ is nilpotent.

\begin{prop}\label{prop1}The dimension of $\Lambda$ equals \[\displaystyle\sum_{\mathfrak{m}\subseteq\mathfrak{a}}[k_{\mathfrak{a}}:\F_2]\dim_{k_{\mathfrak{a}}}\T_{\mathfrak{a}}/(2),\] where the sum runs over all maximal ideals $\mathfrak{a}$ of $\T_2$ containing $\mathfrak{m}$, and $k_{\mathfrak{a}}$ is the residue field corresponding to $\mathfrak{a}$.\end{prop}

\begin{proof}The inclusion of $\T_2^{\text{an}}$ into $\T_2$ induces an inclusion $\T_{\mathfrak{m}}^{\text{an}}$ into $\displaystyle\bigoplus_{\mathfrak{m}\subseteq\mathfrak{a}}\T_{\mathfrak{a}}$, and so the dimension of $\Lambda$ is the dimension of the $\overline{\F}_2$-space of maps $\phi:\displaystyle\bigoplus_{\mathfrak{m}\subseteq\mathfrak{a}}\T_{\mathfrak{a}}\rightarrow\overline{\F}_2$.  Any such map can be split into separate maps $\phi_{\mathfrak{a}}$, and all $\phi_{\mathfrak{a}}$ factor through $\T_{\mathfrak{a}}/(2)$.  So the dimension of $\Lambda$ is\[\dimn_{\overline{\F}_2}\Hom_{\Z_2}(\bigoplus_{\mathfrak{m}\subseteq\mathfrak{a}}\T_{\mathfrak{a}},\overline{\F}_2)=\sum_{\mathfrak{m}\subseteq\mathfrak{a}}\dimn_{\overline{\F}_2}\Hom_{\F_2}(\T_{\mathfrak{a}}/(2),\overline{\F}_2)=\sum_{\mathfrak{m}\subseteq\mathfrak{a}}\dimn_{\F_2}\T_{\mathfrak{a}}/(2)=\sum_{\mathfrak{m}\subseteq\mathfrak{a}}[k_{\mathfrak{a}}:\F_2]\dimn_{k_{\mathfrak{a}}}\T_{\mathfrak{a}}/(2).\]\end{proof}

The trivial lower bound $\dimn_{k_{\mathfrak{a}}}\T_{\mathfrak{a}}/(2)\geq1$ gives a lower bound on the dimension of $\Lambda$.  In the case that $\overline{\rho}$ arising from $\mathfrak{m}$ is totally real and absolutely irreducible, we prove a better bound $\dimn_{k_{\mathfrak{a}}}\T_{\mathfrak{a}}/(2)\geq 2$.  This happens when $\mathfrak{m}$ is $\Q(\sqrt{N})$-dihedral for $N>0$.  Let $J_0(N)$ denote the Jacobian of the modular curve $X_0(N)$, so that $\overline{\rho}$ appears as a subrepresentation of the two torsion points $J_0(N)[2]$.  For some maximal ideal $\mathfrak{a}$ containing $\mathfrak{m}$, let $A=J_0(N)[\mathfrak{a}]$ be the subscheme of points that are killed by $\mathfrak{a}$.  By the main result of \cite{MR1094193}, $A$ is the direct sum of copies of $\overline{\rho}$.  (This holds only when $\overline{\rho}$ is absolutely irreducible, but we cover the reducible case in \ref{sec23} without referring to Proposition \ref{prop2}.)

\begin{prop}\label{prop2}If $\mathfrak{m}$ is a maximal ideal of $\T_2^{\textup{an}}$ for which the corresponding representation $\overline{\rho}$ is totally real, then for any maximal ideal $\mathfrak{a}$ of $\T_2$ containing $\mathfrak{m}$, we have the inequality\[\dimn_{k_{\mathfrak{a}}}\T_{\mathfrak{a}}/(2)\geq 2\cdot\textup{multiplicity of }\overline{\rho}\textup{ inside }A.\]\end{prop}

\begin{proof}Since $\overline{\rho}$ is a representation of the Galois group of a totally real field, we know that the points of $A$ are all real.  Since $A$ also has a $\T_{\mathfrak{a}}$-action with annihilator $\mathfrak{a}$, $A$ is a $k_{\mathfrak{a}}$-vector space, whose dimension is twice the multiplicity of $\overline{\rho}$.  We prove the inequality below, from which the proposition follows quickly.

\begin{lem}\label{ineqlem}If $W$ denotes the Witt vector functor, then \[\dim_{k_{\mathfrak{a}}}(A)\leq \rank_{W(k_{\mathfrak{a}})}(\T_{\mathfrak{a}}).\]\end{lem}
\begin{proof}We follow \cite[Section 3.2]{MR2460912}.  A proposition of Merel states that the real variety $J_0(N)(\mathbb{R})$ is connected if $N$ is prime \cite[Proposition 5]{MR1405312}.  If $g$ is the genus of $X_0(N)$, then we know that $J_0(N)(\mathbb{C})=(\R/\Z)^{2g}$, and therefore $J_0(N)(\mathbb{R})=(\R/\Z)^g$.  And we also know that \[J_0(N)[2](\R)=(\Z/2\Z)^g.\]  Additionally, as we know that $\T_2=\bigoplus_{\mathfrak{a}}\T_{\mathfrak{a}}$, and all $\T_{\mathfrak{a}}$ are free $\Z_2$-modules, say of rank $g(\mathfrak{a})$, we know that \[\sum_{\mathfrak{a}}g(\mathfrak{a})=\rank_{\Z_2}(\T_2)=g.\]A lemma of Mazur shows that the $\mathfrak{a}$-adic Tate module, $\displaystyle\lim_{\longleftarrow}J_0(N)[\mathfrak{a}^i]$, is a $\T_{\mathfrak{a}}$-module of rank $2$ \cite[Lemma 7.7]{MR488287}, and therefore a free $\Z_2$-module of rank $2g(a)$, so $J_0(N)[\mathfrak{a}^{\infty}](\C)=(\Q_2/\Z_2)^{2g(\mathfrak{a})}$.  We therefore know that the $2$-torsion points of this scheme are\[J_0(N)[\mathfrak{a}^{\infty},2](\C)=(\Z/2\Z)^{2g(\mathfrak{a})}.\]

If $\sigma$ acting on $J_0(N)(\C)$ denotes complex conjugation, then $(\sigma-1)^2=2-2\sigma$ kills all $2$-torsion, and $\sigma-1$ itself kills all real points.  So within the scheme $J_0(N)[\mathfrak{a}^{\infty},2](\C)$, applying $\sigma-1$ once kills all real points and maps all points to real points, and so \[\dim_{\Z/2\Z}J_0(N)[\mathfrak{a}^{\infty},2](\R)\geq\frac{1}{2}\dim_{\Z/2\Z}J_0(N)[\mathfrak{a}^{\infty},2](\C)=g(\mathfrak{a}).\]But $J_0(N)[2](\R)$ breaks up into its $\mathfrak{a}^{\infty}$ pieces, $J_0(N)[2](\R)=\bigoplus_{\mathfrak{a}}J_0(N)[\mathfrak{a}^{\infty},2](\R)$.  Taking dimensions on both sides gives \[g=\sum_{\mathfrak{a}}\dim_{\Z/2\Z}J_0(N)[\mathfrak{a}^{\infty},2](\R)\geq\sum_{\mathfrak{a}}g(\mathfrak{a})=g,\] so equality must hold everywhere.

Since all points of $A=J_0(N)[\mathfrak{a}]$ are real, we find that \[\dim_{\Z/2\Z}A\leq\dim_{\Z/2\Z}J_0(N)[\mathfrak{a}^{\infty},2](\R)=g(\mathfrak{a})=\rank_{\Z_2}(\T_{\mathfrak{a}}).\]Dividing both sides by $[k_{\mathfrak{a}}:\Z/2\Z]=\text{rank}(W(k_{\mathfrak{a}})/\Z_2)$, we have the result.\end{proof}

Returning to the proof of Proposition \ref{prop2}, we therefore know that \[\dimn_{k_{\mathfrak{a}}}\T_{\mathfrak{a}}/(2)=\dimn_{W(k_{\mathfrak{a}})}\T_{\mathfrak{a}}\geq2\cdot\text{multiplicity of }\overline{\rho}.\]
\end{proof}

For reference, we recall a theorem of Deligne that describes the characteristic $0$ representation $\rho$ restricted to the decomposition group at $2$:
\begin{thm}[{\cite[Theorem~2.5]{MR1176206}}]\label{deligne}If $\rho_f$ is an ordinary $2$-adic representation corresponding to a weight $2$ level $\Gamma_0(N)$ form $f$, then $\rho_f|_{D_2}$, the restriction of $\rho_f$ to the decomposition group at a prime above $2$, is of the shape\[\rho|_{D_2}\sim\begin{pmatrix}\chi\lambda^{-1}&*\\ 0&\lambda\end{pmatrix}\] for $\lambda$ the unramified character $G_{\Q_2}\rightarrow\Z_2^{\times}$ taking $\Frob_2$ to the unit root of $X^2-a_2X+2$, and $\chi$ is the $2$-adic cyclotomic character.\end{thm}

 \section{$N\equiv 1\bmod 8$}\label{sec2}
\subsection{$K=\Q(\sqrt{N})$}\label{sec21}
\begin{thm}If $N\equiv1\bmod 8$, and $\mathfrak{m}$ is a maximal ideal of $\T_2^{\textup{an}}(N)$ that is $\Q(\sqrt{N})$-dihedral, then $\dim S_2(N)_{\mathfrak{m}}\geq 4$.\end{thm}
\begin{proof}
Let $K=\Q(\sqrt{N})$ and denote the fixed field of the kernel of $\overline{\rho}$ as $L$.  In this $K$, the prime $(2)$ factors as $\mathfrak{p}\mathfrak{q}$ for distinct $\mathfrak{p}$ and $\mathfrak{q}$, and $\overline{\rho}$ must be unramified at $2$ so $\Frob_2$, as a conjugacy class containing $\Frob_{\mathfrak{p}}$ and $\Frob_{\mathfrak{q}}$, must lie in $\Gal(L/K)$.  Moreover, $\overline{\rho}$ must be semisimple at $2$, because if $\overline{\rho}=\Ind_K^{\Q}\overline{\chi}$ for $\overline{\chi}$ a character of the unramified extension $\Gal(L/K)$, then $\overline{\rho}|_{\Gal(L/K)}=\overline{\chi}\oplus\overline{\chi}^g$ where $g\in\Gal(L/\Q)\backslash\Gal(L/K)$.

Theorem~\ref{deligne} and this semisimplicity statement tell us that the decomposition group at $2$ in the ${}\bmod 2$ representation looks like $\begin{pmatrix}\lambda^{-1}&0\\0&\lambda\end{pmatrix}$, because the cyclotomic character is always $1\bmod 2$.  So we find that the polynomial $\det(x\Id_2-\overline{\rho})$ has coefficients that are unramified at $2$, and $T_2$ is a root of $P(x):=\det(x\Id_2-\overline{\rho}(\Frob_2))$.  There are thus three cases: either $P$ has no roots already in $k:=\T^{\text{an}}/\mathfrak{m}$, or it has distinct roots lying in $k$, or it has a repeated root.

If $P$ has no roots in $k$, then $[k_{\mathfrak{a}}:k]\geq 2$ for $\mathfrak{a}$ the extension of $\mathfrak{m}$, so Propositions~\ref{prop1} and \ref{prop2} say that the dimension of the space is at least \[ [k_{\mathfrak{a}}:\F_2]\dimn_{k_{\mathfrak{a}}}\T_{\mathfrak{a}}/(2)\geq[k_{\mathfrak{a}}:k]\dimn_{k_{\mathfrak{a}}}\T_{\mathfrak{a}}/(2)\geq 2\cdot 2=4.\]

If $P$ has distinct roots in $k$, then there are at least $2$ extensions of $\mathfrak{m}$ to $\T_2$.  Namely, if $x_1$ and $x_2$ are lifts of the roots of $P$ to $\T^{\text{an}}_\mathfrak{m}$, the two ideals $\mathfrak{a}_1=(\mathfrak{m}, T_2-x_1)$ and $\mathfrak{a}_2=(\mathfrak{m}, T_2-x_2)$ are two maximal ideals.  So in this case the dimension is at least\[[k_{\mathfrak{a}_1}:\F_2]\dimn_{k_{\mathfrak{a}_1}}\T_{\mathfrak{a}_1}/(2)+[k_{\mathfrak{a}_2}:\F_2]\dimn_{k_{\mathfrak{a}_2}}\T_{\mathfrak{a}_2}/(2)\geq\dimn_{k_{\mathfrak{a}_1}}\T_{\mathfrak{a}_1}/(2)+\dimn_{k_{\mathfrak{a}_2}}\T_{\mathfrak{a}_2}/(2)\geq 2+2=4.\]

Finally, suppose $P$ has a double root.  There is at least one maximal ideal $\mathfrak{a}$ of $\T_2$ above $\mathfrak{m}$.  Because we know that $\overline{\rho}|_{D_2}$ is semisimple with determinant $1$, the double root must be $1$ and $\overline{\rho}|_{D_2}$ is trivial.  Then Wiese proves that since all dihedral representations arise from Katz weight 1 modular forms (as Wiese proves in \cite{MR2054983}),  the multiplicity of $\overline{\rho}$ in $A$ is $2$ \cite[Corollary 4.5]{MR2336635}.  In this case the dimension is at least \[[k_{\mathfrak{a}}:\F_2]\dimn_{k_{\mathfrak{a}}}\T_{\mathfrak{a}}/(2)\geq\dimn_{k_{\mathfrak{a}_1}}\T_{\mathfrak{a}_1}/(2)\geq 2\cdot\text{multiplicity of }\overline{\rho}\geq4.\]
\end{proof}
\subsection{$K=\Q(\sqrt{-N})$}\label{sec22}
\begin{thm}\label{thm22}If $N\equiv1\bmod 8$, and $\mathfrak{m}$ is a maximal ideal of $\T_2^{\textup{an}}(N)$ that is $\Q(\sqrt{-N})$-dihedral, then $\dim S_2(N)_{\mathfrak{m}}\geq 2^e$ where $2^e=\big|\Cl(K)[2^{\infty}]\big|$.\end{thm}
\begin{proof}
We first recall a well-known proposition of genus theory:

\begin{prop}\label{genusthy}Let $K=\Q(\sqrt{-d})$ be an imaginary quadratic field with $d>0$ squarefree.

\begin{enumerate}[label=(\alph*), font=\upshape]
\item The $\F_2$-dimension of the $2$-torsion of the class group of $K$ is one less than the number of primes dividing the discriminant $\Delta_{K/\Q}$.
\item If $d\equiv5\bmod 8$ is a prime, then the $2$-part of the class group of $K$ is cyclic of order $2$.
\item If $d\equiv1\bmod 8$ is a prime, then the $2$-part of the class group of $K$ is cyclic of order at least $4$.\end{enumerate}\end{prop}
A proof of the final two parts can be found as \cite[Proposition 4.1]{MR2129709}.

We return to the case $N\equiv 1\bmod 8$.  Proposition~\ref{genusthy} tells us that the $2$-part of the class group is cyclic so there is an unramified $\Z/(2^e)$-extension $L'/K$, say $\Gal(L'/K)=\langle g\rangle$ with $g^{2^e}=\Id$.  If we as before denote by $L$ the fixed field of the kernel of $\overline{\rho}$, and we let $M=L\cdot L'$, the character $\overline{\chi}$ of $\Gal(L/K)$ whose induction equals $\overline{\rho}$ can be extended to a character $\overline{\chi}':\Gal(M/K)\rightarrow \overline{\F}_2[x]/(x^{2^e}-1)$ given by mapping $g$ to $x$.  This can be done because $L\cap L'=K$, because $[L:K]$ is odd, say $[L:K]=n$, and $[L':K]$ is a power of $2$.  Then the induction of $\overline{\chi}$ to $\overline{\rho}$ also extends from $\overline{\chi}'$ to $\overline{\rho}':\Gal(M/\Q)\rightarrow\GL_2(\overline{\F}_2[x]/(x^{2^e}-1))$.  We will prove this representation is modular by describing a $q$-expansion with coefficients in $\overline{\Z}_2[x]/(x^{2^e}-1)$ whose reduction ${}\bmod 2$ gives the desired Frobenius traces as coefficients, and proving that the expansion is modular via the embeddings of this coefficient ring into $\C$.  Then by the $q$-expansion principle we will have the result.

Let us suppose we have chosen a $2^e$th root of unity $\eta:=\zeta_{2^e}$ inside $\overline{\Z}_2$.  We may lift $\overline{\chi}$ to a character $\chi: \Gal(L/K)\rightarrow \Z_2^{\text{ur}}$.  We may therefore also lift $\overline{\chi}'$ to a character $\chi':\Gal(M/K)\rightarrow\Z_2^{\text{ur}}[x]/(x^{2^e}-1)$.  We may tensor with $\Q_2$, and identifying $\Q_2^{\text{ur}}[x]/(x^{2^e}-1)$ with $\bigoplus_{i=0}^e\Q_2^{\text{ur}}(\zeta_{2^i})$ by sending $x$ to $\eta^{2^{e-i}}$ gives us $e+1$ representations \[\chi_i: \Gal(M/K)\rightarrow\Q_2^{\text{ur}}(\zeta_{2^i})^{\times}\text{ and }\rho_i: \Gal(M/\Q)\rightarrow\GL_2(\Q_2^{\text{ur}}(\zeta_{2^i})).\]These are all finite image odd representations whose coefficients are algebraic and therefore may be compatibly embedded in $\C$.  All twists of $\rho_i$ are dihedral or nontrivial cyclic, and therefore all have analytic $L$-functions.  So by the converse theorem of Weil and Langlands (see \cite[Theorem~1]{MR0450201}, for instance), each $\rho_i$ corresponds to a weight $1$ eigenform $f_i$ with modulus equal to the conductor of the representation and nebentypus equal to its determinant.  Here, the conductor is $4N$ and the nebentypus is the nontrivial character of $\Gal(K/\Q)$.  This nebentypus, because $K$ has discriminant $4N$, is the character $\lambda_{4N}:=\lambda_4\lambda_N$ where $\lambda_4$ and $\lambda_N$ are the nontrivial order $2$ characters of $(\Z/4\Z)^{\times}$ and $(\Z/N\Z)^{\times}$; $\lambda_{4N}(p)=1$ if and only if $\Frob_p$ is the identity in $\Gal(K/\Q)$ if and only if $p$ splits in $K$.

Each $f_i$ is a simultaneous eigenvector for the entirety of the weight $1$ Hecke algebra $\T(4N)$, with coefficients in $\Q_2^{\text{ur}}(\zeta_{2^i})$, so by returning to $\Q_2^{\text{ur}}[x]/(x^{2^e}-1)$ we obtain a weight $1$ form $f$ with coefficients in this ring, which is therefore an eigenform by multiplicity $1$ results.  We can easily check that the traces of the representation $\rho'=\Ind_K^{\Q}\chi': \Gal(M/\Q)\rightarrow\GL_2(\Q_2^{\text{ur}}[x]/(x^{2^e}-1))$ correspond to the coefficients of $f$, and so since $\chi'$ and therefore $\rho$ are defined over $\Z_2^{\text{ur}}[x]/(x^{2^e}-1)$, $f$ also has coefficients in $\Z_2^{\text{ur}}[x]/(x^{2^e}-1)$.

Now we take the characteristic $0$ form $f$ and multiply by a modular form of weight $1$, level $\Gamma_1(4N)$ and nebentypus $\chi_{4N}$ whose $q$-expansion is congruent to $1\bmod2$.  That will give us a weight $2$ level $\Gamma_0(4N)$ form whose ${} \bmod 2$ reduction is equal to the $q$-expansion of a form coming from $\overline{\rho}'$.  We find such a form:

\begin{lem}The $q$-expansion $\sum_{m, n\in\Z}q^{m^2+Nn^2}$ describes a modular form $g$ in $S_1(4N,\Z_2,\chi_{4N})$.\end{lem}
\begin{proof}This follows from properties of the Jacobi theta function $\vartheta(\tau)=\displaystyle\sum_{k\in\mathbb{Z}}q^{k^2}$, but we give a different proof.  Let $\delta$ range over all characters of the class group $H$ of $K$, or equivalently over all unramified characters of $\Gal(\overline{\Q}/K)$.  By Weil-Langlands, $\Ind_K^{\Q}\delta$ as a representation of $G_{\Q}$ gives us a weight $1$ modular form.  The determinant of this induction is always equal to $\chi_{K/\Q}$, and the conductor is always equal to $4N$.  For two of the characters, $\delta$ trivial and $\delta$ the nontrivial character of $\Gal(K(i)/K)$, $\Ind_K^{\Q}\delta$ is reducible and the weight $1$ modular forms are the Eisenstein series\[E^{\chi_{4N},_{\mathbf{1}}}(q)=L(\chi_{4N}, 0)/2+\sum_{m=1}^{\infty}q^m\sum_{d\text{ odd, }d|m}(-1)^{(d-1)/2}\left(\frac{d}{N}\right)\]and\[E^{\chi_N, \chi_4}(q)=\sum_{m=1}^\infty q^m\sum_{d\text{ odd, }de=m}(-1)^{(d-1)/2}\left(\frac{e}{N}\right)\]respectively.  The constant term of the former is, by the class number formula, equal to $h(-N)/2$.  Otherwise, the forms are cusp forms $f_{\delta}$ with no constant term.

\begin{lem}The $q$-expansion of $f_{\delta}$ is given by $f_{\delta}=\displaystyle\sum_{m\geq 1}q^m\sum_{I\subseteq\mathcal{O}_K: N(I)=m}\delta(I)$.\end{lem}
\begin{proof}If $p$ is a prime inert in $K$, then there is no $I$ with $N(I)=p$.  In the representation $\Ind_K^{\Q}\delta$, $\Frob_p$ is antidiagonal, so it has trace $0$, which is therefore the Hecke eigenvalue.  So for $p$ inert in $K$, the coefficient is correct.  If $p=\mathfrak{p}_1\mathfrak{p}_2$, then $\sum_{I\subseteq\mathcal{O}_K: N(I)=p}\delta(I)=\delta(\mathfrak{p}_1)+\delta(\mathfrak{p}_2)$, and the trace of $\Frob_p$ in the representation is also $\delta(\mathfrak{p}_1)+\delta(\mathfrak{p}_2)$ because the restriction of $\Ind_K^{\Q}\delta$ to $G_K$ is diagonal with characters $\delta$ and $\delta^g$ for $g$ a lift of the nontrivial element of $\Gal(K/\Q)$ and $\delta^g(h)$ meaning $\delta(ghg^{-1})$.  Since all primes over $p$ are conjugate, $\delta^g(\mathfrak{p}_1)=\delta(\mathfrak{p}_2)$ and so the trace of $\Frob_p$ is $\delta(\mathfrak{p}_1)+\delta(\mathfrak{p}_2)$ as we needed.

If $p=N$, the ideal over $N$ is principal, and so splits completely in $M/K$; on inertia invariants, therefore, its Frobenius is trivial and the coefficient of $q^N$ is $1$, as is necessary since $\delta((\sqrt{-N}))=1$ because $\delta$ is a character of the class group.  And if $p=2$, the ideal $\mathfrak{p}$ over $2$ has order $2$ in the class group.  The inertia subgroup for some prime over $2$ in $M$ is generated by some lift of the nontrivial element of $\Gal(K/\Q)$, and the decomposition group is the product of this group with the subgroup of $\Gal(M/K)$ corresponding to the class of $\mathfrak{p}$.  And so on inertia invariants, the eigenvalue of the decomposition group is the eigenvalue of $\Frob_{\mathfrak{p}}$, which is $\delta(\mathfrak{p})$.  So the coefficient for $q^2$ is correct as well.

Finally, we can check using multiplicativity of both Hecke operators and the norm map, as well as the formula for the Hecke operators $T_{p^k}$, that the coefficients of $q^m$ for composite $m$ are as described also.
\end{proof}

We consider $\sum_{\delta}f_\delta$, cusp forms with their multiplicity (stemming from $\delta$ and $\delta^{-1}$ giving the same form) and the Eisenstein series once.  By independence of characters, for each ideal $I$ where $\delta(I)=1$ for all $\delta$, that is $I$ is in the identity of the class group, the corresponding term in the sum is $h(-N)$, and for each other nonzero ideal $I$, the term vanishes in the sum.  The sum is thus \begin{align*}L(\chi_{4N}, 0)/2+h(-N)\sum_{0\neq I=(\alpha)}q^{N(I)}&=h(-N)/2+\frac{h(-N)}{|\mathcal{O}_K^{\times}|}\sum_{0\neq\alpha=a+b\sqrt{-N}\in \mathcal{O}_K}q^{N(\alpha)}
\\ &=\frac{h(-N)}{2}\left(1+\sum_{(0, 0)\neq (a, b)\in\Z}q^{a^2+Nb^2}\right).\end{align*}Dividing by $h(-N)/2$ gives the required form, which we call $g$.
\end{proof}

It turns out that there is a form (not an eigenform) of level $\Gamma_1(N)$ which lifts the Hasse invariant.  It is a linear combination of the Eisenstein series $E^{\epsilon,\mathbf{1}}(q)$ for $\epsilon$ ranging over all $2^{v_2(N)}$-order characters of $\Z/N\Z^{\times}$, and has the correct nebentypus when reduced because all $2$-power roots of unity are $1\bmod{}$ the maximal ideal over $2$ in $\Z[\eta]$.  This form is described by MathOverflow user Electric Penguin in \cite{228596}.

So we take $fg$ and reduce the coefficients mod the maximal ideal over $2$ and get a form $h\in S_2(\overline{\F}_2[x]/(x^{2^e}-1),\Gamma_0(4N))$.  We know that $h$ remains an eigenform because for odd primes, $p\equiv1\bmod2$ so increasing the weight doesn't change the Hecke action on the coefficients, and for $2$ increasing the weight does not change the action of $U_2$ on $q$-expansions.  Because $h$ is an eigenform, if $\T(4N)$ now represents the Hecke algebra acting on weight $2$ forms of level $\Gamma_0(4N)$, we get a ring homomorphism $\overline{\gamma}: \T(4N)\rightarrow\overline{\F}_2[x]/(x^{2^e}-1)$.  The image of this map tensored with $\overline{\F}_2$ is the entirety of $\overline{\F}_2[x]/(x^{2^e}-1)$: we have prime ideals of $K$ in all elements of the class group, so if $\mu$ is some nonzero element in the image of $\overline{\chi}$ not equal to $1$, then both $\mu x+\mu^{-1}x^{-1}$ and $\mu x^{-1}+\mu^{-1}x$ are in the image of $\overline{\gamma}$, so that \[\mu^{-1}(\mu x^{-1}+\mu^{-1}x)+\mu(\mu x+\mu^{-1}x^{-1})=(\mu^2+\mu^{-2})x\] is in the $\overline{\F}_2$ vector space generated by the image of $\overline{\gamma}$, and hence $x$ is also.  And since $\overline{\gamma}$ is a ring homomorphism, all powers of $x$ lie in the filled out image.

As described in \cite[Section 3.3]{MR2460912}, we may find a representation \[G_{\Q}\rightarrow\GL_2(\overline{\F}_2[x]/(x^{2^e}-1)),\] in the following way: we let $\mathfrak{a}'$ denote the kernel of $\T(4N)\xrightarrow{\overline{\gamma}}\overline{\F}_2[x]/(x^{2^e}-1)\xrightarrow{x\mapsto1}\overline{\F}_2$, and we let $\T(4N)_{\mathfrak{a}'}$ denote the completion of $\T(4N)$ with respect to that ideal.  The Galois action on $J_0(4N)[\mathfrak{a}']$ breaks into isomorphic $2$-dimensional representations $G_{\Q}\rightarrow\GL_2(\T(4N)/\mathfrak{a}')$, and Carayol constructs a lift $G_{\Q}\rightarrow\GL_2(\T(4N)_{\mathfrak{a}'})$ \cite[Theorem 3]{MR1279611}.  We pushforward this map along $\T(4N)_{\mathfrak{a}'}\rightarrow\overline{\F}_2[x]/(x^{2^e}-1)$ which also has full image to get a representation $G_{\Q}\rightarrow\overline{\F}_2[x]/(x^{2^e}-1)$.  It's clear that this representation is isomorphic to $\overline{\rho}'=\Ind_{K}^{\Q}\overline{\chi}'$ by looking at traces.  So $\overline{\rho}'$ is modular of level $\Gamma_0(4N)$.

We know that $h$ is an eigenform for $U_2$, and the operator $U_2$ lowers the level from $4N$ to $2N$.  So $h=U_2h$ is an eigenform of level $\Gamma_0(2N)$.  We recall the level lowering theorem of Calegari and Emerton; here $A$ is an Artinian local ring of residue field $k$ of characteristic $2$.

\begin{thm}[{\cite[Theorem 3.14]{MR2460912}}]\label{calem}If $\rho:G_{\Q}\rightarrow\GL_2(A)$ is a modular Galois representation of level $\Gamma_0(2N)$, such that\begin{enumerate}\item$\overline{\rho}$ is (absolutely) irreducible,\item$\overline{\rho}$ is ordinary and ramified at $2$, and\item$\rho$ is finite flat at $2$,\end{enumerate}then $\rho$ arises from an $A$-valued Hecke eigenform of level $N$.\end{thm}
Our $\overline{\rho}'$, pushed forward through the map $\overline{\F}_2[x]/(x^{2^e}-1)\rightarrow\overline{\F}_2$ and restricting to its true image, is irreducible, ordinary and ramified.  All that remains in order to apply the theorem is to check that $\overline{\rho}'$ is finite flat at $2$.  It's enough to show this after restricting to $\Gal(\overline{\Q}_2/\Q_2^{\text{ur}})$.  But the representation has only degree two ramification, so the image of $\Gal(\overline{\Q}_2/\Q_2^{\text{ur}})$ is order $2$.  And furthermore, it's easy to see that it arises as the generic fiber of $D^{\oplus2^e}$ over $\Z_2^{\text{ur}}$, where $D$ is the nontrivial extension of $\Z/2\Z$ by $\mu_2$ discussed in \cite[Proposition 4.2]{MR488287}, represented for example by $\Z_2[x,y]/(x^2-x,y^2+2x-1)$ with comultiplication \[x\rightarrow x_1+x_2-2x_1x_2\text{ and }y\rightarrow y_1y_2-2x_1x_2y_1y_2.\]  So we may apply Theorem~\ref{calem}, and deduce that our modular form $h$ is a modular form of level $N$.

We have therefore constructed a surjective map $\T_{\mathfrak{m}}\otimes_{\Z_2}\overline{\F}_2\rightarrow\overline{\F}_2[x]/(x^{2^e}-1)$, so the $\overline{\F}_2$-dimension of $S_2(\Gamma_0(N),\overline{\F}_2)_{\mathfrak{m}}$ must be at least $2^e$.  Note that Proposition~\ref{genusthy} shows that this dimension is at least $4$.
\end{proof}

\subsection{$\mathfrak{m}$ is reducible}\label{sec23}
\begin{thm}If $N\equiv1\bmod 8$, and $\mathfrak{m}$ is a maximal ideal of $\T_2^{\textup{an}}(N)$ for which $\overline{\rho}_{\mathfrak{m}}$ is reducible, then $\dim S_2(N)_{\mathfrak{m}}\geq \frac{h(-N)^{\textup{even}}-2}{2}$.\end{thm}
\begin{proof}
We know that $\mathfrak{m}\subseteq\T^{\text{an}}$ is generated by $T_{\ell}$ and $2$ for all primes $\ell\not|2N$.  In \cite[Corollary~4.9]{MR2129709} and the discussion after Proposition 4.11,  Calegari and Emerton prove that $\T^{\text{an}}_{\mathfrak{m}}/(2)$ must be isomorphic to $\F_2[x]/(x^{2^{e-1}})$, where $2^e=h(-N)^{\text{even}}$.  They accomplish this by setting up a deformation problem, namely deformations of $(\overline{V}, \overline{L}, \overline{\rho})$ where $\overline{\rho}$ is the mod-2 representation $\left(\begin{smallmatrix}1&\phi\\ 0&1\end{smallmatrix}\right)$, $\phi$ is the additive character $G_{\Q}\rightarrow \F_2$ that arises as the nontrivial character of $\Gal(\Q(i)/\Q)$, and $\overline{L}$ is a line in $\overline{V}$ not fixed by $G_{\Q}$.  With the conditions set on the deformation, they find that it is representable by some $\Z_2$-algebra $R$.

Next, they prove an $R=\T$-type theorem, namely that $R=\T$ where $\T$ is the completion at the Eisenstein ideal of the Hecke algebra acting on all modular forms of level $\Gamma_0(N)$, including the Eisenstein series.  Finally they study $R/2$ which represents the deformation functor to characteristic $2$ rings, and show that if $\rho^{\text{univ}}$ is the universal deformation, then $\rho^{\text{univ}}$ factors through the largest unramified $2$-extension of $K$.  This combined with their fact that a map $R\rightarrow \F_2[x]/(x^n)$ can be surjective if and only if $n\leq 2^{e-1}$ proves that $R/2=\F_2[x]/(x^{2^{e-1}})$.

Therefore, the same holds for the Eisenstein Hecke algebra $\T/2$.  So we know that $\T$ is a free $\Z_2$-module of rank $\frac{h(-N)^{\text{even}}}{2}$.  But we may split off a one-dimensional subspace corresponding to the Eisenstein series, so that the cuspidal Hecke algebra $\T^{\text{an}}_{\mathfrak{m}}$ has rank one less, and therefore has rank $\frac{h(-N)^{\text{even}}}{2}-1$.  (In fact, the full Hecke algebra is determined also, because in any reducible ${} \bmod 2$ representation, $T_2$ and $U_N$ must both map to $1$, as $U_N$ is unipotent and $T_2$ maps to the image of Frobenius under a ${} \bmod 2$ character unramified at every prime not equal to $2$.  But there are no nontrivial such characters.)  And therefore the dimension of the space $S_2(N)_{\mathfrak{m}}$ is the dimension of the space $\Hom(\T^{\text{an}}_{\mathfrak{m}},\overline{\F}_2)$, which is dimension $\frac{h(-N)^{\text{even}}}{2}-1$, as desired.
\end{proof}
\cite{kedlaya2019} partially prove this theorem using \cite{MR2129709}, doing the case of $N\equiv 9\bmod16$.  As we see, the method works equally well for $N\equiv1\bmod16$.  The only difference between the two cases is that \cite{MR2129709} prove that for $N\equiv9\bmod16$, the Hecke algebra $\T_{\mathfrak{m}}^{\text{an}}$ is a discrete valuation ring, and therefore a domain, but that plays no role here.
\section{$N\equiv5\bmod 8$}\label{sec3}
\subsection{$K=\Q(\sqrt{N})$}\label{sec31}
\begin{thm}If $N\equiv5\bmod 8$, and $\mathfrak{m}$ is a maximal ideal of $\T_2^{\textup{an}}(N)$ that is $\Q(\sqrt{N})$-dihedral, then $\dim S_2(N)_{\mathfrak{m}}\geq 4$.\end{thm}
\begin{proof}Because $2$ is inert in $\Q(\sqrt{N})$, we know that $\overline{\rho}|_{D_2}$ is of size $2$.  Then the image of $\overline{\rho}$ is a subgroup of a $2$-Sylow subgroup of $\GL_2(\overline{\F}_2)$, and therefore is isomorphic to an upper-triangular idempotent representation $\overline{\rho}|_{D_2}\simeq \left(\begin{smallmatrix}1& *\\0&1\end{smallmatrix}\right)$.  If we compare to Theorem~\ref{deligne}, we find that in an eigenform for all $T_p$ including $T_2$ that corresponds to this representation, $a_2=1$.  So the three methods of section \ref{sec21} do not work.

Recall Proposition~\ref{prop2} that says if the representation $\overline{\rho}$ was totally real, then $\dimn_{k_{\mathfrak{a}}}\T_{\mathfrak{a}}/(2)\geq 2\cdot\text{multiplicity of }\overline{\rho}$, so if this multiplicity is at least $2$, we're done.  So we assume that $\overline{\rho}$ occurs once in $J_0(N)[\mathfrak{a}]$.  However, we know by \cite[Theorem 4.4]{MR2336635} that since $\overline{\rho}$ comes from a Katz modular form of weight $1$ and level $N$, and the multiplicity of $\overline{\rho}$ on $J_0(N)[\mathfrak{a}]$ is $1$, that the multiplicity of $\overline{\rho}$ in $J_0(N)[\mathfrak{m}]$ is $2$.  So by the proof as in Proposition~\ref{prop2}, we know the dimension of $\T_{\mathfrak{m}}/(2)$ has dimension at least twice $2$, or dimension $4$, and so $\dimn S_2(N)_{\mathfrak{m}}\geq 4$ as required.
\end{proof}

\subsection{$K=\Q(\sqrt{-N})$}\label{sec32}
\begin{thm}If $N\equiv5\bmod 8$, and $\mathfrak{m}$ is a maximal ideal of $\T_2^{\textup{an}}(N)$ that is $\Q(\sqrt{-N})$-dihedral, then $\dim S_2(N)_{\mathfrak{m}}\geq 2$.\end{thm}
This follows in a similar way to Theorem \ref{thm22}.  Proposition~\ref{genusthy} proves that the $2$ part of the class group of $K$ is order $2$, so applying the results of section \ref{sec22} proves the theorem in this case.  The only difficulties are in verifying the conditions of Theorem~\ref{calem}; that is, $\overline{\rho}$ is absolutely irreducible, ordinary, and ramified, and $\rho$ itself is finite flat at $2$.  It's clear that the first three conditions hold, and the final condition holds because $\Q_2^{\text{ur}}(\sqrt{-N})=\Q_2^{\text{ur}}(i)$ even though $N\equiv 5\bmod 8$.  So the group scheme in this case is the same as the group scheme in section \ref{sec22}, and we have verified all necessary conditions.
\section{$N\equiv3\bmod4$}\label{sec4}
\subsection{$K=\Q(\sqrt{N})$}\label{sec41}
\begin{thm}If $N\equiv3\bmod 4$, and $\mathfrak{m}$ is a maximal ideal of $\T_2^{\textup{an}}(N)$ that is $\Q(\sqrt{N})$-dihedral, then $\dim S_2(N)_{\mathfrak{m}}\geq 2$.\end{thm}
\begin{proof}We let $\mathfrak{a}$ be a prime of $\T_2$ containing $\mathfrak{m}$.  Then again recalling Proposition~\ref{prop2}, since $K$ and therefore $\overline{\rho}$ are totally real, we calculate that the dimension is at least \[\dimn_{k_{\mathfrak{a}}}\T_{\mathfrak{a}}/(2)\geq2\cdot\text{multiplicity of }\overline{\rho}\geq 2\]as required.
\end{proof}
\subsection{$K=\Q(\sqrt{-N})$}\label{sec42}
\begin{thm}If $N\equiv3\bmod 4$, and $\mathfrak{m}$ is a maximal ideal of $\T_2^{\textup{an}}(N)$ that is $\Q(\sqrt{-N})$-dihedral, then $\dim S_2(N)_{\mathfrak{m}}\geq 4$.\end{thm}
\begin{proof}
This was shown in \cite[Proposition 14]{kedlaya2019} using essentially the same method as we use in sections \ref{sec22} and \ref{sec32}.  The only differences are that $K/\Q$ is unramified at $2$ so the Artin conductor of $\overline{\rho}'$ is $N$, not $4N$, so no level-lowering is required; and that we obtain a second eigenspace from our modular form $f$ coming from the reduction of $f^2$.
\end{proof}
\section{The effect of $U_N$}
In none of our proofs did we ever exploit the fact that $U_N$ is not defined to be in $\T_2^{\text{an}}$ as we did with $T_2$, and the following gives an explanation why.

\begin{lem}There is an inclusion $U_N\in\T_2^{\textup{an}}$, so $\T_2=\T_2^{\textup{an}}[T_2]$.\end{lem}
\begin{proof}Since $\T_2^{\text{an}}=\bigoplus_{\mathfrak{m}}\T_{\mathfrak{m}}^{\text{an}}$, it suffices to prove that $U_N\in\T_{\mathfrak{m}}^{\text{an}}$ for each maximal ideal $\mathfrak{m}$.  Let \[\overline{\rho}=\overline{\rho}_{\mathfrak{m}}:G_{\Q}\rightarrow\GL_2(\T_{\mathfrak{m}}^{\text{an}}/\mathfrak{m})\subseteq\GL_2(\overline{\F}_2)\]denote the residual representation associated to $\mathfrak{m}$.  If $\overline{\rho}$ is not irreducible, then it is Eisenstein.  The Eisenstein ideal $\mathfrak{I}\subseteq\T_2$ is generated by $1+\ell-T_{\ell}$ for $\ell\neq N$ and by $U_N-1$.  Let $\mathfrak{a}=(2,\mathfrak{I})$ denote the corresponding maximal ideal of $\T_2$.  By \cite[Proposition 17.1]{MR488287}, the ideal $\mathfrak{a}$ is actually generated by $\eta_{\ell}:=1+\ell-T_{\ell}$ for a suitable good prime $\ell\neq 2, N$.  But this implies that $\T_{\mathfrak{m}}^{\text{an}}=\T_{\mathfrak{a}}$ and that $U_N$ (and $T_2$) lie in $\T_{\mathfrak{m}}^{\text{an}}$.  Hence we assume that $\overline{\rho}$ is irreducible.

If $\overline{\rho}$ is irreducible but not absolutely irreducible, then its image would have to be cyclic of degree prime to $2$.  Since the image of inertia at $N$ is unipotent it has order dividing $2$.  Thus this would force $\overline{\rho}$ to be unramified at $N$.  There are no nontrivial odd cyclic extensions of $\Q$ ramified only at $2$, and thus this does not occur, and we may assume that $\overline{\rho}$ is absolutely irreducible.

Tate proved in \cite{Tate} the following theorem:
\begin{thm}[Tate]\label{tates}Let $G$ be the Galois group of a finite extension $K/\Q$ which is unramified at every odd prime.  Suppose there is an embedding $\rho: G\xhookrightarrow{}\SL_2(k)$, where $k$ is a finite field of characteristic $2$.  Then $K\subseteq\Q(\sqrt{-1},\sqrt{2})$ and $\Tr\rho(\sigma)=0$ for each $\sigma\in G$.\end{thm}

If $\overline{\rho}$ is unramified at $N$, then $\det\overline{\rho}$ is a character of odd order unramified outside $2$, which by Kronecker-Weber must be trivial, so $\overline{\rho}$ maps to $\SL_2(k)$.  We may apply Theorem \ref{tates} to determine that $\overline{\rho}$ has unipotent image, which therefore is not absolutely irreducible.  Hence we may assume that $\overline{\rho}$ is ramified at $N$.  By local-global compatibility at $N$, the image of inertia at $N$ of $\overline{\rho}$ is unipotent.  Because it is nontrivial, it thus has image of order exactly $2$.

Let $\{f_i\}$ denote the collection of $\overline{\Q}_2$-eigenforms such that $\overline{\rho}_{f_i}=\overline{\rho}$.  Associated to each $f_i$ is a field $E_i$ generated by the eigenvalues $T_l$ for $l\neq 2, N$.  There exists a corresponding Galois representation:\[\rho:G_{\Q}\rightarrow\GL_2(\T_{\mathfrak{m}}^{\text{an}}\tensor\Q)=\prod\GL_2(E_i).\]The traces of $\rho$ at Frobenius elements land inside $\T_{\mathfrak{m}}^{\text{an}}$, and hence the traces of all elements land inside $\T_{\mathfrak{m}}^{\text{an}}$.  By a result of Carayol, there exists a choice of basis so that $\rho$ is valued inside $\GL_2(\T_{\mathfrak{m}}^{\text{an}})$; that is, there exists a free $\T_{\mathfrak{m}}^{\text{an}}$-module of rank $2$ with a Galois action giving rise to $\rho$.  Each representation $\rho_{f_i}$ has the property that, locally at $N$, the image of inertia is unipotent.  In particular, $\rho|_{G_{\Q_N}}$ is tamely ramified.  Let $\langle\sigma, \tau\rangle$ denote the Galois group of the maximal tamely ramified extension of $\Q_N$, where $\sigma$ is a lift of Frobenius and $\tau$ a pro-generator of tame inertia, so $\sigma\tau\sigma^{-1}=\tau^N$.  We claim that there exists a basis of $(\T_{\mathfrak{m}}^{\text{an}})^2$ such that\[\overline{\rho}|_{G_{\Q_N}}=\left(\begin{matrix}1&1\\ 0&1\end{matrix}\right).\]
Note, first of all, that it is true modulo $\mathfrak{m}$ by assumption (because $\overline{\rho}$ is ramified).  Choose a lift $e_2\in(\T_{\mathfrak{m}}^{\text{an}})^2$ of a vector which is not fixed by $\overline{\rho}(\tau)$, and then let $e_1=(\rho(\tau)-1)e_2$.  Since the reduction of $e_1$ and $e_2$ generate $(\T_{\mathfrak{m}}^{\text{an}}/\mathfrak{m})^2$, by Nakayama's lemma they generate $(\T_{\mathfrak{m}}^{\text{an}})^2$.  Finally we have $(\rho(\tau)-1)^2=0$ since $(\rho_{f_i}(\tau)-1)^2=0$ for each $i$.

Now consider the image of $\sigma$.  Writing\[\rho(\sigma)=\begin{pmatrix}a&b\\ c&d\end{pmatrix}\in\GL_2(\T_{\mathfrak{m}}^{\text{an}}),\]the condition that $\rho(\sigma)\rho(\tau)=\rho(\tau)^N\rho(\sigma)$ forces $c=0$.  But then if\[\rho(\sigma)=\begin{pmatrix}*&*\\ 0&x\end{pmatrix}\in\GL_2(\T_{\mathfrak{m}}^{\text{an}}),\]then for every specialization $\rho_{f_i}$, the action of Frobenius on the unramified quotient is $x$.  But for each $\rho_{f_i}$, the action of Frobenius on the unramified quotient is the image $U_N(f_i)$ of $U_N$.  Hence this implies that $x=U_N$, and thus that $U_N\in\T_{\mathfrak{m}}^{\text{an}}$.\end{proof}
\section{Acknowledgments}
The author would like to thank Frank Calegari for the many useful conversations leading toward the proofs in sections $1$ through $4$, and from whom the proof of Lemma 5.1 was communicated.
\nocite{*}
\bibliographystyle{alpha}
\bibliography{LaTeX2bib}
\end{document}